\documentclass{article}

\usepackage{a4}
\usepackage{paralist}

\title{Robust space-time finite element error estimates\\
for parabolic distributed optimal control problems \\ 
with energy regularization} 
\author{Ulrich~Langer\footnote{Institute of Computational Mathematics,
    Johannes Kepler University Linz, 
    Altenberger Stra{\ss}e 69, 4040 Linz, Austria,
    Email: ulanger@numa.uni-linz.ac.at},
    \; Olaf~Steinbach\footnote{Institut f\"{u}r Angewandte Mathematik,
    Technische Universit\"{a}t Graz, Steyrergasse 30, 8010 Graz, Austria,
    Email: o.steinbach@tugraz.at}, 
    \; Huidong~Yang\footnote{Institute of Science and Technology Austria
      (ISTA), Am Campus 1, 3400 Klosterneuburg, Austria,
      Email: huidong.yang@ist.ac.at} 
}  

\date{\today}

\usepackage{amsmath}
\usepackage{amsthm}
\usepackage{amsfonts}
\usepackage{booktabs}
\usepackage{graphicx}
\usepackage{diagbox}

\usepackage{multirow}
\usepackage{hyperref}

\usepackage[ulem=normalem,draft]{changes}

\newtheorem{theorem}{Theorem}
\newtheorem{lemma}{Lemma}
\newtheorem{cor}{Corollary}

\numberwithin{equation}{section} 

\begin{document}

\maketitle

\begin{abstract}
We consider space-time tracking optimal control problems 
for  linear para\-bo\-lic initial boundary value problems
that are given in the space-time cylinder $Q = \Omega \times (0,T)$, 
and that are controlled by the right-hand side $z_\varrho$ from
the Bochner space $L^2(0,T;H^{-1}(\Omega))$.
So it is natural to replace the usual $L^2(Q)$ norm regularization by 
the energy regularization in the $L^2(0,T;H^{-1}(\Omega))$ norm.
We derive a priori estimates for the error 
$\|\widetilde{u}_{\varrho h} - \overline{u}\|_{L^2(Q)}$ 
between the computed state $\widetilde{u}_{\varrho h}$ and the desired
state $\overline{u}$ in terms of the regularization parameter $\varrho$
and the space-time finite element mesh-size $h$,
and depending on the regularity of the desired state $\overline{u}$.
These estimates lead to the optimal choice $\varrho = h^2$.
The approximate state $\widetilde{u}_{\varrho h}$ is computed by means of
a space-time finite element method using piecewise linear
and continuous basis functions
on completely unstructured simplicial meshes for $Q$. The theoretical
results are quantitatively illustrated 
by a series of numerical examples in two and three space dimensions.
\end{abstract} 

\begin{keywords}
  Parabolic optimal control problems, energy regularization, \\
  space-time finite element methods, error estimates.
\end{keywords}

\begin{msc}
49J20,  
49M05,  
35K20,  
65M60,  
65M15   
\end{msc}

%
%
\section{Introduction}
\label{sec:Introduction}
As in \cite{LSY:LangerSteinbachTroeltzschYang:2021c}, we  consider the
minimization of the  space-time tracking cost functional
\begin{equation}\label{Parabolic minimization problem}
  {\mathcal{J}}(u_\varrho,z_\varrho) = \frac{1}{2} \int_0^T \int_\Omega
  [u_\varrho(x,t) - \overline{u}(x,t)]^2 \, dx \, dt +
  \frac{1}{2} \, \varrho \, \| z_\varrho \|^2_{L^2(0,T;H^{-1}(\Omega))}
\end{equation}
with respect to the state $u_\varrho$ and the control $z_\varrho$
subject to the model parabolic initial boundary value problem
\begin{equation}\label{Parabolic PDE}
  \left. \begin{array}{rclcl}
           \partial_t u_\varrho(x,t) - \Delta_x u_\varrho(x,t)
           & = & z_\varrho(x,t) & \quad &
           \mbox{for} \; (x,t) \in Q := \Omega \times (0,T), \\[1mm]
           u_\varrho(x,t) & = & 0 & &
           \mbox{for} \; (x,t) \in \Sigma:=\partial\Omega
           \times (0,T), \\[1mm]
           u_\varrho(x,0) & = & 0 & & \mbox{for} \; x \in \overline{\Sigma}_0 :=
                                      \overline{\Omega} \times \{ 0 \},
         \end{array} \right \}
\end{equation}
where  $\overline{u} \in L^2(Q)$ is the given desired state (target),
$\partial_t$ denotes the partial time derivative, 
$\Delta_x$ is the spatial Laplace operator,
$\Omega \subset {\mathbb{R}}^n$, $n=1,2,3$, is the spatial domain that
is assumed to be bounded and Lipschitz,
$T>0$ is a given time horizon, and $\varrho > 0$ is a suitably
chosen regularization parameter.
The standard setting of such kind of optimal control problems uses the 
regularization in $L^2(Q)$ instead of $L^2(0,T;H^{-1}(\Omega))$;
see, e.g., the books \cite{LSY:BorziSchulz:2011a,
  LSY:HinzePinnauUlbrichUlbrich:2009a,LSY:Troeltzsch:2010a},
and the references given therein. 
The energy regularization, as the regularization in
$L^2(0,T;H^{-1}(\Omega))$ is also called,
permits controls $z_\varrho$ from the space
$L^2(0,T;H^{-1}(\Omega))$ that is larger than $L^2(Q)$, and admits more
concentrated controls. Such kind of controls that are concentrated around
hypersurfaces play an important role in electromagnetics in form of
thinly wound coils and magnets. Moreover, the space
$L^2(0,T;H^{-1}(\Omega))$ is the natural space for the source term in the
variational formulation of the initial boundary value problem
\eqref{Parabolic PDE}, at least, in the Hilbert space setting; see, e.g.,  
\cite{LSY:Lions:1968a} or \cite{LSY:Zeidler:1990a} for solvability results. 
In the literature, there are other regularization techniques aiming at specific 
properties of the control such as sparsity  and  directional sparsity.
We refer the reader to the recent survey article \cite{LSY:Casas:2017a} where a 
comprehensive overview of the literature on this topic is given.

Since the state equation \eqref{Parabolic PDE} in its variational form
has a unique solution 
$u_\varrho \in X:= \{u \in Y:=L^2(0,T;H^{1}_0(\Omega)):
\partial_t u \in Y^*, \; u = 0 \;\text{on} \; \Sigma_0\}$,
for every given right-hand side $z_\varrho \in Y^*:= L^2(0,T;H^{-1}(\Omega))$,
the corresponding optimal control problem
\eqref{Parabolic minimization problem}-\eqref{Parabolic PDE}
also has a unique solution $(u_\varrho, z_\varrho) \in X \times Y^*$ 
that can be computed by solving the first-order optimality system or
the reduced first-order optimality system
where the control is eliminated by the gradient equation.
The unique solvability of the state equation can also be shown by the 
Banach--Ne\u{c}as--Babu\u{s}ka theorem
as it was done in \cite{LSY:Steinbach:2015a}.
This theorem can also be used to show
well-posedness of the reduced first-order optimality system 
as it was done in \cite{LSY:LangerSteinbachTroeltzschYang:2021c}.
Now the optimal control problem \eqref{Parabolic minimization problem}-\eqref{Parabolic PDE} 
can be approximately solved by discretizing the reduced optimality system.
Following \cite{LSY:LangerSteinbachTroeltzschYang:2021c}, we discretize
the reduced optimality system by means of a real space-time finite
element method working on fully unstructured, but shape regular 
simplicial space-time meshes into which the space-time cylinder $Q$
is decomposed.
In \cite{LSY:LangerSteinbachTroeltzschYang:2021c}, the authors showed a
discrete inf-sup condition for the bilinear form arising from the
variational formulation of the reduced optimality system. 
Once a discrete inf-sup condition is proven, one can easily derive the
corresponding estimates for the finite element discretization error
$u_\varrho - \widetilde{u}_{\varrho h}$ and $p_\varrho - \widetilde{p}_{\varrho h}$
in the corresponding norms, where $\widetilde{u}_{\varrho h}$ and
$\widetilde{p}_{\varrho h}$ are the finite element solutions 
to the reduced first-order optimality system approximating the state
$u_\varrho$ and the co-state (adjoint) $p_\varrho$, respectively.

In this paper, we are investigating the error between the computed finite
element solution $\widetilde{u}_{\varrho h}$ and the desired state
$\overline{u}$, where
we use continuous, piecewise linear finite element basis functions.
This error is obviously of primary interest since one 
wants to know how well $\widetilde{u}_{\varrho h}$ approximates
$\overline{u}$ in advance.
More precisely, we derive estimates for the $L^2(Q)$ norm of this error
in terms of $\varrho$ and $h$, and depending on the smoothness of
the target $\overline{u}$ that is assumed to belong to $H^s(Q)$ for
some $s \in (0,2]$.
In particular, we admit discontinuous targets that are important in many
practical applications.
These estimates lead to the optimal choice $\varrho = h^2$ in all cases. 
For elliptic optimal control problems with energy regularization, i.e.,
in $H^{-1}(\Omega)$,
error estimates for $\|\overline{u} - {u}_{\varrho}\|_{L^2(Q)}$ and 
$\| \overline{u} - \widetilde{u}_{\varrho h}\|_{L^2(Q)}$
were recently derived in
\cite{LSY:NeumuellerSteinbach:2021a}
and \cite{LSY:LangerSteinbachYang:2022a}, respectively.
It is interesting that, in the elliptic case,  ${u}_{\varrho}$ solves 
the singularly perturbed reaction-diffusion equation
$- \varrho \Delta u_\varrho + u_\varrho = \overline{u} $ 
in $\Omega$ with homogeneous Dirichlet conditions on the boundary
$\partial \Omega$, also known as differential filter in fluid
mechanics \cite{LSY:John:2016a},
whereas, in the parabolic case, ${u}_{\varrho}$ solves a similar
singularly perturbed problem,
but with a more complicated space-time operator of the form
$B^*A^{-1}B$ replacing $-\Delta$,
where $B: X \to Y^*$ is nothing but the state (parabolic) operator,
and $A: Y \to Y^*$ represents the spatial Laplacian $-\Delta_x$;
see Sections~\ref{sec:AbstractOCP} and \ref{sec:ParabolicOCP} for a more
detailed discussion.

The reminder of this paper is organized as follows:
Section~\ref{sec:AbstractOCP} deals with the formulation of an abstract
optimal control problem,
and the corresponding error estimates between the desired state and
the discrete state based on the exact state Schur complement equation. 
In Section~\ref{sec:ParabolicOCP}, we consider a model parabolic
distributed optimal control problem with energy regularization,
and derive estimates for the $L^2(Q)$
error between the desired state $\overline{u}$ and the finally
computed state $ \widetilde{u}_{\varrho h} $ from the
perturbed state Schur complement equation for the coupled optimality system. 
Several numerical tests in two and three space dimensions are discussed in
Section~\ref{sec:NumericalResults}. 
Finally, some conclusions are drawn in Section~\ref{sec:ConclusionsOutlook},
and we also discuss some future research topics.  


\section{Abstract optimal control problems}
\label{sec:AbstractOCP}
Let $X \subset H \subset X^*$ and $Y \subset H \subset Y^*$
be Gelfand triples of Hilbert spaces, where $X^*, Y^*$ are
the duals of $X, Y$ with respect to $H$. Let $ A : Y \to Y^*$
and $B: X \to Y^*$ be bounded linear operators, i.e.,
\begin{equation}\label{Bounded abstract}
  \| A v \|_{Y^*} \leq c_2^A \, \| v \|_Y \quad \forall v \in Y,
  \quad
  \| B u \|_{Y^*} \leq c_2^B \, \| u \|_X \quad \forall u \in X.
\end{equation}
We assume that $A$ is self-adjoint and elliptic in $Y$, and that $B$
satisfies an inf-sup condition, i.e., there exist positive constants
$c_1^A$ and $c_1^B$ such that
\begin{equation}\label{inf-sup abstract}
  \langle A v,v \rangle_H \geq c_1^A \, \| v \|_Y^2 \quad \forall v \in Y,
  \quad \sup\limits_{0 \neq v \in Y}
  \frac{\langle Bu,v\rangle_H}{\| v \|_Y} \geq c_1^B \, \| u \|_X \quad
  \forall u \in X .
\end{equation}
In addition, we assume that the dual to $B$ operator $B^*: Y \to X^*$ is injective. 
Then, due to Lax--Milgram's and Banach--Ne\v{c}as--Babu\v{s}ka's
theorems (see, e.g.,\cite{LSY:ErnGuermond:2004a}), 
$A:Y \to Y^*$ and $B:X \to Y^*$ are isomorphisms. Therefore, 
\begin{equation} \label{Y^*Norm}
  \| z \|_{Y^*} = \sqrt{\langle A^{-1} z , z \rangle_H} \quad
  \mbox{for} \; z \in Y^* 
\end{equation}
defines a norm in $Y^*$ that is equivalent to the standard supremum norm.

We now consider the abstract minimization problem 
to find the minimizer $(u_\varrho,z_\varrho) \in X \times Y^*$
of the functional
\begin{equation}\label{Abstract minimization problem}
  {\mathcal{J}}(u_\varrho,z_\varrho) =
  \frac{1}{2} \, \| u_\varrho - \overline{u} \|_H^2 +
  \frac{1}{2} \, \varrho \, \| z_\varrho \|^2_{Y^*} \quad
  \mbox{subject to} \; Bu_\varrho=z_\varrho ,
\end{equation}
when $\overline{u} \in H$ is given, and $\varrho \in {\mathbb{R}}_+$
is some regularization parameter. 
For the time being, our particular interest is focused on the behavior
of $\|u_\varrho - \overline{u}\|_H$ as $\varrho \to 0$.
The minimizer $(u_\varrho,z_\varrho)$ of (\ref{Abstract minimization problem})
is determined as the unique solution of the optimality system,
see, e.g., \cite{LSY:LangerSteinbachTroeltzschYang:2021c},
\begin{equation}\label{Abstract optimality system}
  B u_\varrho = z_\varrho, \qquad
  B^* p_\varrho = u_\varrho - \overline{u} , \qquad
  p_\varrho + \varrho A^{-1} z_\varrho = 0 .
\end{equation}
%
Eliminating
the control $z_\varrho \in Y^*$ and the adjoint
variable $p_\varrho \in Y$ 
results in 
the
operator equation to find $u_\varrho \in X$ such that
\begin{equation}\label{Abstract reduced reduced optimality system}
  \varrho B^* A^{-1} B u_\varrho + u_\varrho = \overline{u}
  \quad \mbox{in} \; X^* .
\end{equation}  
Let us introduce the operator $ S := B^* A^{-1} B : X \to X^*$,
for which we have the following result:
\begin{lemma}\label{Lemma S abstrakt}
  There hold the inequalities
  \[
    \langle S u , u \rangle_H \geq c_1^S \, \| u \|_X^2 \quad \mbox{and} \quad
    \| S u \|_{X^*} \leq c_2^S \, \| u \|_X \quad \mbox{for all} \; u \in X
  \]
  with constants
  \[
    c_1^S = c_1^A \left( \frac{c_1^B}{c_2^A} \right)^2 
    \quad \mbox{and} \quad 
    c_2^S = \frac{[c_2^B]^2}{c_1^A} \, .
  \]
\end{lemma}
\begin{proof}
  For 
  arbitrary, but fixed $u \in X$,
  we define $\overline{p} = A^{-1} B u$ to
  obtain
  \[
    \langle S u , u \rangle_H =
    \langle A^{-1} B u , Bu \rangle_H =
    \langle A \overline{p} , \overline{p} \rangle_H \geq
    c_1^A \, \| \overline{p} \|_Y^2 \, .
  \]
  From the inf-sup condition \eqref{inf-sup abstract} we further conclude
  \[
    c_1^B \, \| u \|_X \leq \sup\limits_{0 \neq v \in Y}
    \frac{\langle Bu , v \rangle_H}{\| v \|_Y} =
    \sup\limits_{0 \neq v \in Y}
    \frac{\langle A \overline{p} , v \rangle_H}{\| v \|_Y} \leq
    \| A \overline{p} \|_{Y^*} \leq c_2^A \, \| \overline{p} \|_Y .
  \]
  This gives
  \[
    \langle S u , u \rangle_H \geq c_1^A \, \| \overline{p} \|^2_Y \geq
    c_1^A \left( \frac{c_1^B}{c_2^A} \right)^2 \| u \|_X^2 =
    c_1^S \, \| u \|_X^2 .
  \]
  To prove the second estimate, we consider
  \[
    c_1^A \, \| \overline{p} \|_Y^2 \leq
    \langle A \overline{p} , \overline{p} \rangle_H =
    \langle B u , \overline{p} \rangle_H \leq
    \| B u \|_{Y^*} \| \overline{p} \|_Y \leq
    c_2^B \, \| u \|_X \| \overline{p} \|_Y,
  \]
  i.e.,
  \[
    \| \overline{p} \|_Y \leq \frac{c_2^B}{c_1^A} \, \| u \|_X \, .
  \]
  With this we finally obtain
  \begin{eqnarray*}
    \| S u \|_{X^*}
    & = & \sup\limits_{0 \neq v \in X}
          \frac{\langle S u , v \rangle_H}{\| v \|_X} =
          \sup\limits_{0 \neq v \in X}
          \frac{\langle A^{-1} B u , B v \rangle_H}{\| v \|_X} \\
    & = & \sup\limits_{0 \neq v \in X}
          \frac{\langle \overline{p} , B v \rangle_H}{\| v \|_X} \leq
          \sup\limits_{0 \neq v \in X}
          \frac{\| \overline{p} \|_Y  \| B v \|_{Y^*}}{\| v \|_X} \\
    & \leq & c_2^B \, \| \overline{p} \|_Y \leq
             \frac{[c_2^B]^2}{c_1^A} \, \| u \|_X = c_2^S \, \| u \|_X \, .
  \end{eqnarray*}
\end{proof}

\noindent
As a consequence of Lemma \ref{Lemma S abstrakt} we also have
\[
  \langle S u , u \rangle_H \leq \| S u \|_{X^*} \| u \|_X \leq
  c_2^S \, \| u \|_X^2 ,
\]
i.e.,
\[
  \| u \|_S^2 := \langle S u , u \rangle_H =
  \langle A^{-1} B u , Bu \rangle_H 
\]
defines an equivalent norm in $X$ satisfying the norm equivalence inequalities
\begin{equation}\label{norm equivalence}
  c_1^S \, \| u \|_X^2 \leq \| u \|_S^2 \leq c_2^S \, \|u \|_X^2 \quad
  \mbox{for all} \; u \in X .
\end{equation}
Now we consider the abstract operator equation to find $u_\varrho \in X$
such that
\begin{equation}\label{Abstract operator equation}
\varrho S u_\varrho + u_\varrho = \overline{u} \quad \mbox{in} \; X^*,
\end{equation}
and its equivalent variational formulation
\begin{equation}\label{Abstract variational formulation}
  \varrho \, \langle S u_\varrho,v \rangle_H + \langle u_\varrho , v \rangle_H =
  \langle \overline{u},v \rangle_H \quad \mbox{for all} \; v \in X .
\end{equation}
Since $S$ induces an equivalent norm in $X$, unique solvability of
(\ref{Abstract variational formulation}) follows.
\begin{lemma}
  For the unique solution $u_\varrho \in X$ of the variational formulation
  \eqref{Abstract variational formulation}, there hold the estimates
  \begin{equation}\label{Abstract HH SH}
    \| u_\varrho \|_H \leq \| \overline{u} \|_H \quad \mbox{and} \quad
    \varrho \, \| u_\varrho \|^2_S \leq \| \overline{u} \|^2_H .
  \end{equation}
\end{lemma}

\begin{proof}
  For the particular choice $v=u_\varrho$ within the variational
  formulation \eqref{Abstract variational formulation}, we obtain
  \[
    \varrho \, \| u_\varrho \|_S^2 + \| u_\varrho \|^2_H
    =  \varrho \, \langle S u_\varrho , u_\varrho \rangle_H +
    \langle u_\varrho , u_\varrho \rangle_H 
    =  \langle \overline{u} , u_\varrho \rangle_H
    \leq \| \overline{u} \|_H \| u_\varrho \|_H ,
  \]
  from which we conclude
  \[
  \| u_\varrho \|_H \leq \| \overline{u} \|_H
  \]
  as well as
  \[
    \varrho \, \| u_\varrho \|^2_S \leq \| \overline{u} \|_H
    \| u_\varrho \|_H \leq \| \overline{u} \|^2_H .
  \]
\end{proof}

\noindent
Analogously to \cite[Theorem 3.2]{LSY:NeumuellerSteinbach:2021a} we can state
the following estimates, which depend on the regularity of the given
target $\overline{u} \in H$.

\begin{lemma}\label{Abstract Lemma error rho}
  Let $u_\varrho \in X$ be the unique solution of the variational formulation
  \eqref{Abstract variational formulation}. For $\overline{u} \in H$ there
  holds
  \begin{equation}\label{Abstract error H H}
    \| u_\varrho - \overline{u} \|_H \leq \| \overline{u} \|_H,
  \end{equation}
  while for $\overline{u} \in X$ the following estimates hold true:
  \begin{equation}\label{Abstract error H X}
    \| u_\varrho - \overline{u} \|_H \leq \varrho^{1/2} \,
    \| \overline{u} \|_S ,
  \end{equation}
  \begin{equation}\label{Abstract error X X}
    \| u_\varrho - \overline{u} \|_S \leq \| \overline{u} \|_S .
  \end{equation}
  If in addition $S\overline{u} \in H$ is satisfied for $\overline{u} \in X$, 
  \begin{equation}\label{Abstract error H Su}
    \| u_\varrho - \overline{u} \|_H \leq \varrho \, \| S \overline{u} \|_H
  \end{equation}
  as well as
  \begin{equation}\label{Abstract error S Su}
    \| u_\varrho - \overline{u} \|_S \leq \varrho^{1/2} \, \| S \overline{u} \|_H
  \end{equation}
  follow.
\end{lemma}

\begin{proof}
  From the variational formulation
  \eqref{Abstract variational formulation} and for the particular
  test function $v=u_\varrho$, we obtain
  \[
    \varrho \, \| u_\varrho \|_S^2 =
    \varrho \, \langle S u_\varrho , u_\varrho \rangle_H =
    \langle \overline{u} - u_\varrho , u_\varrho \rangle_H =
    \langle \overline{u} - u_\varrho , \overline{u} \rangle_H -
    \langle \overline{u} - u_\varrho , \overline{u} - u_\varrho \rangle_H,
  \]
  which gives
  \[
    \varrho \, \| u_\varrho \|_S^2 + \| u_\varrho - \overline{u} \|_H^2 =
    \langle \overline{u} - u_\varrho , \overline{u} \rangle_H \leq
    \| \overline{u} - u_\varrho \|_H \| \overline{u} \|_H,
  \]
  i.e., (\ref{Abstract error H H}) follows.

  When assuming $\overline{u} \in X$, we can choose
  $v = \overline{u} - u_\varrho \in X$ as test function in
  \eqref{Abstract variational formulation} to conclude
  \begin{eqnarray} \nonumber
    \| \overline{u} - u_\varrho \|_H^2
    & = & \langle \overline{u} - u_\varrho ,
          \overline{u} - u_\varrho \rangle_H \\
    & = &  \varrho \, \langle S u_\varrho , \overline{u} - u_\varrho \rangle_H
    \label{Abstract Error H X Zwischenschritt}\\
    & = & \varrho \, \langle S \overline{u} , \overline{u} - u_\varrho
          \rangle_H -
          \varrho \, \langle S(\overline{u} - u_\varrho),
          \overline{u}-u_\varrho \rangle_H, \nonumber
  \end{eqnarray}
  i.e.,
  \[
    \varrho \, \| \overline{u} - u_\varrho \|_S^2 +
    \| \overline{u} - u_\varrho \|^2_H
    = \varrho \, \langle S \overline{u} , \overline{u} - u_\varrho \rangle_H
    \leq \varrho \, \| \overline{u} \|_S \| \overline{u} - u_\varrho \|_S \, .
  \]
  In a first step this gives \eqref{Abstract error X X},
  \[
  \| u_\varrho - \overline{u} \|_S \leq \| \overline{u} \|_S .
  \]
  With this we further obtain
  \[
    \| u_\varrho - \overline{u} \|^2_H
    \leq \varrho \, \| \overline{u} \|_S \| \overline{u} - u_\varrho \|_S
    \leq \varrho \, \| \overline{u} \|_S^2,
  \]
  i.e., (\ref{Abstract error H X}) follows.

  If, for $\overline{u} \in X$, we have in addition $S\overline{u} \in H$,
  from the estimate \eqref{Abstract Error H X Zwischenschritt}, we also
  conclude
  \[
    \varrho \, \| \overline{u} - u_\varrho \|_S^2 +
    \| \overline{u} - u_\varrho \|^2_H
    = \varrho \, \langle S \overline{u} , \overline{u} - u_\varrho \rangle_H
    \leq \varrho \, \| S \overline{u} \|_H \| u_\varrho - \overline{u} \|_H,
  \]
  from which \eqref{Abstract error H Su} follows. Finally, the estimates
  \[
    \varrho \, \| \overline{u} - u_\varrho \|_S^2 
    \leq \varrho \, \| S \overline{u} \|_H \| u_\varrho - \overline{u} \|_H
    \leq \varrho^2 \, \| S \overline{u} \|_H^2
  \]
  imply \eqref{Abstract error S Su}.
\end{proof}

\noindent
Based on the estimates as given in Lemma \ref{Abstract Lemma error rho}
and in the case of the particular application we have in mind,
we can derive more general estimates which are based on interpolation
arguments in a scale of Sobolev spaces. This will be discussed later
in more detail. 

For some conforming approximation space $X_h \subset X$, we now consider
the Galerkin variational formulation of
\eqref{Abstract variational formulation}, i.e.,
find $u_{\varrho h} \in X_h$ such that
\begin{equation}\label{Abstract Galerkin}
  \varrho \, \langle S u_{\varrho h} , v_h \rangle_H +
  \langle u_{\varrho h} , v_h \rangle_H = \langle \overline{u} , v_h \rangle_H
  \quad \forall \, v_h \in X_h .
\end{equation}
Using again standard arguments, we conclude unique solvability of
(\ref{Abstract Galerkin}), and the following Cea type a priori error
estimate,
\begin{equation}\label{Abstract Cea}
    \| u_\varrho - u_{\varrho h} \|_H \leq \inf\limits_{v_h \in X_h}
    \sqrt{\varrho \, \| u_\varrho - v_h \|^2_S + \| u_\varrho - v_h \|^2_H} .
\end{equation}
As a particular application of \eqref{Abstract Cea} we obtain, when
choosing $v_h=0$, and using \eqref{Abstract HH SH},
\[
  \| u_\varrho - u_{\varrho h} \|_H^2
  \leq \varrho \, \| u_\varrho \|^2_S + \| u_\varrho \|^2_H \, \leq \,
  2 \, \| \overline{u} \|_H^2 .
\]
Now, using \eqref{Abstract error H H}, we conclude the abstract
error estimate
\begin{equation}\label{Error H H}
  \| u_{\varrho h} - \overline{u} \|_H \leq
  \| u_\varrho - \overline{u} \|_H + \| u_\varrho - u_{\varrho h} \|_H
  \leq (1+\sqrt{2}) \, \| \overline{u} \|_H .
\end{equation}
when assuming $\overline{u} \in H$ only.

  
\section{Parabolic distributed optimal control problem}
\label{sec:ParabolicOCP}
The parabolic optimal control problem
\eqref{Parabolic minimization problem}-\eqref{Parabolic PDE} as given in 
the introduction is obviously a special case of the abstract optimal
control problem \eqref{Abstract minimization problem}. Indeed, 
in view of the abstract setting, we have
$H := L^2(Q)$, $Y := L^2(0,T;H^1_0(\Omega))$, and 
\[
  X := \{ u \in W(0,T) : u=0 \; \mbox{on} \; \Sigma_0\}, 
\]
with 
$W(0,T) :=  \{ u \in Y: \partial_t u \in Y^* = L^2(0,T;H^{-1}(\Omega)) \}$.
The related norms in $Y$, $X$, and $Y^*$ are given by
\[
  \| v \|_Y := \| \nabla_x v \|_{L^2(Q)}, \;
  \| u \|_X := \sqrt{ \| u \|^2_Y +
    \| \partial_t u \|^2_{Y^*}} , \; \text{and} \;
  \| \partial_t u \|_{Y^*} =
    \| \nabla_x w_u \|_{L^2(Q)},
\]
respectively, where $w_u \in Y$ is the unique solution of the variational problem
\[
  \langle \nabla_x w_u , \nabla_x v \rangle_{L^2(Q)} =
  \langle \partial_t u , v \rangle_Q \quad \forall \, v \in Y.
\]
For later use, we will prove the following embedding:

\begin{lemma}
  For $ u \in X \cap H^1(Q)$ there holds
  \begin{equation}\label{Norm X H1}
    \| u \|_X \leq \max \{ \sqrt{c_F} , 1 \} \, \| u \|_{H^1(Q)} 
  \end{equation}
  with the constant $c_F>0$ from the spatial Friedrichs inequality
  in $H^1_0(\Omega)$,
  \begin{equation}
   \label{FriedrichsInequality}
    \int_\Omega [v(x)]^2 \, dx \, \leq \ c_F \int_\Omega |\nabla_x v(x)|^2 \, dx
    \quad \forall \, v \in H^1_0(\Omega) .
   \end{equation}
\end{lemma}
\begin{proof}
  Recall that we can write
  \[
    \| u \|^2_X = \| \partial_t u \|_{Y^*}^2 + \| \nabla_x u \|^2_{L^2(Q)},
  \]
  and since  $\partial_t u \in L^2(Q)$ for $u \in H^1(Q)$,
  we can bound $\| \partial_t u \|_{Y^*}$ as follows:
  \[
    \| \partial_t u \|_{Y^*} =
    \sup\limits_{0 \neq v \in Y}
    \frac{\langle \partial_t u , v \rangle_Q}{\| v \|_{Y}}
    \leq     \sup\limits_{0 \neq v \in Y}
    \frac{\| \partial_t u \|_{L^2(Q)} \| v \|_{L^2(Q)}}
    {\| \nabla_x v \|_{L^2(Q)}} \leq \sqrt{c_F} \, \| \partial_t u \|_{L^2(Q)} .
  \]
  Here we have used the Friedrichs inequality 
  \[
    \| v \|_{L^2(Q)}^2 = \int_0^T \| v(t) \|^2_{L^2(\Omega)} \, dt \leq
    c_F \int_0^T \| \nabla_x v(t) \|^2_{L^2(\Omega)} \, dt =
    c_F \, \| \nabla_x v \|^2_{L^2(Q)}
  \]
  that holds for all $v \in Y = L^2(0,T;H_0^1(\Omega))$ due
  to \eqref{FriedrichsInequality}. Hence, the estimates
  \[
    \| u \|_X^2 \leq c_F \, \| \partial_t u \|^2_{L^2(Q)} +
    \| \nabla_x u \|^2_{L^2(Q)} \leq \max \{ c_F , 1 \} \, \| u \|^2_{H^1(Q)} 
  \]
  follow. 
\end{proof}

\noindent
The variational formulation of the state equation \eqref{Parabolic PDE}
can now be written in the form: Find $u_\varrho \in X$ such that
\begin{equation*}
\int_0^T \int_\Omega \Big[ \partial_t u_\varrho(x,t) \, v(x,t) +
\nabla_x u_\varrho(x,t) \cdot \nabla_x v(x,t) \Big] \, dx \, dt 
=
\int_0^T \int_\Omega z_\varrho(x,t)  \, v(x,t)  \, dx \, dt 
\end{equation*}
for all $v \in Y$, where the first term in the bilinear form 
and the right-hand side must be understood as duality pairing
between $Y^*$ and $Y$.
This variational formulation can be rewritten as 
operator equation $B u_\varrho = z_\varrho$ in $Y^* = L^2(0,T;H^{-1}(\Omega))$.
The operator $B : X \to Y^*$ is therefore defined by the variational identity
\begin{equation}
\label{Definition_B}
  \langle B u , v \rangle_Q = \int_0^T \int_\Omega
  \Big[ \partial_t u(x,t) \, v(x,t) + \nabla_x u(x,t) \cdot
  \nabla_x v(x,t) \Big] \, dx \, dt
\end{equation}  
for all $u \in X$ and  $v \in Y$, while $A : Y \to Y^*$ is given as
\begin{equation}
\label{Definition_A}
  \langle A w , v \rangle_Q = \int_0^T \int_\Omega \nabla_x w(x,t)
  \cdot \nabla_x v(x,t) \, dx \, dt, \quad \forall  \, w,v \in Y .
\end{equation}   
We obviously have $c_1^A=c_2^A=1$.
Following \cite{LSY:Steinbach:2015a,LSY:SteinbachZank:2020a},
the operator $B : X \to Y^*$ is bounded,
\[
  \langle B u ,v \rangle_Q \leq \sqrt{2} \, \| u \|_X \| v \|_Y
  \quad \forall \, u \in X, \; v \in Y,
\]
and satisfies the inf-sup condition
\[
  \frac{1}{\sqrt{2}} \, \| u \|_X \leq
  \sup\limits_{0 \neq v \in Y} \frac{\langle Bu,v \rangle_Q}{\| v \|_Y}
  \quad \forall \, u \in X,
\]
i.e., $c_1^B=1/\sqrt{2}$ and $c_2^B=\sqrt{2}$. Hence we obtain the
statements of Lemma \ref{Lemma S abstrakt} with $c_1^S=1/2$ and $c_2^S=2$.
With these definitions, the reduced first-order optimality system can be 
written in the following operator form:
Find $(u_\varrho, p_\varrho) \in X \times Y$ such that
\begin{equation}
\label{rOS} 
\begin{pmatrix} \varrho^{-1} A & B\\ 
                B^* & -I 
\end{pmatrix}
\begin{pmatrix} p_\varrho  \\  u_\varrho
\end{pmatrix} = \begin{pmatrix} 0\\ -\overline{u} \end{pmatrix}
\quad \text{in}  \quad Y^* \times X^*,
\end{equation}
from which the control $z_\varrho = - \varrho^{-1} A p_\varrho$ can be computed;
cf. also \eqref{Abstract optimality system} and
\eqref{Abstract reduced reduced optimality system}.

For the Galerkin formulation \eqref{Abstract Galerkin}, we introduce
a conforming finite element space $X_h = S_h^1(Q) \cap X \subset X$
of piecewise linear and continuous basis functions which are defined
with respect to some admissible decomposition of the space-time
domain $Q$ into shape regular simplicial finite elements of mesh
width $h$; see, e.g., \cite{LSY:BrennerScott:2008a}.
Then the finite element approximation of 
\eqref{Abstract variational formulation} reads to
find $u_{\varrho h} \in X_h$ such that
\begin{equation}\label{VF parabolic S FEM}
  \varrho \, \langle B^* A^{-1} B u_{\varrho h} , v_h \rangle_Q +
  \langle u_{\varrho h} , v_h \rangle_{L^2(Q)} =
  \langle \overline{u} , v_h \rangle_{L^2(Q)}
\end{equation}
is satisfied for all $v_h \in X_h$.

\begin{theorem}
\label{Theorem1}
Assume $\overline{u} \in [L^2(Q),X]_s \cap H^s(Q)$ for $s \in [0,1)$
or $\overline{u} \in X \cap H^s(Q)$ for $s \in [1,2]$.
For the unique solution $u_{\varrho h} \in X_h$ of \eqref{VF parabolic S FEM},
the finite element error estimate
\begin{equation}
\label{Th1Estimate}
\| u_{\varrho h} - \overline{u} \|_{L^2(Q)} \, \leq \, c \, h^s \,
\| \overline{u} \|_{H^s(Q)} 
\end{equation}
holds provided that $\varrho = h^2$.
\end{theorem}

\begin{proof}
  For $\overline{u} \in L^2(Q)$, we can write the error estimate
  \eqref{Error H H} as
  \[
    \| u_{\varrho h} - \overline{u} \|_{L^2(Q)} \leq (1+\sqrt{2}) \,
    \| \overline{u} \|_{L^2(Q)} .
  \]
  Due to $X \subset H^1(Q)$, we now assume
  $\overline{u} \in X \cap H^1(Q)$ for which we can write the error estimate
  \eqref{Abstract Cea} as
  \begin{eqnarray*}
    \| u_\varrho - u_{\varrho h} \|_{L^2(Q)}^2
    & \leq & \inf\limits_{v_h \in X_h} \Big[
             \varrho \, \| u_\varrho - v_h \|_S^2 +
             \| u_\varrho - v_h \|_{L^2(Q)}^2 \Big] \\
    & & \hspace*{-3cm} \leq \, 2 \left[
             \varrho \, \| u_\varrho - \overline{u} \|_S^2 +
             \| u_\varrho - \overline{u} \|^2_{L^2(Q)} +
             \inf\limits_{v_h \in X_h} \Big[
             \varrho \, \| \overline{u} - v_h \|^2_S +
             \| \overline{u} - v_h \|^2_{L^2(Q)} \Big]
        \right] \\
    & & \hspace*{-3cm} \leq \, 4 \, \varrho \, \| \overline{u} \|_S^2 + 2
             \inf\limits_{v_h \in X_h} \Big[
             \varrho \, \| \overline{u} - v_h \|^2_S +
             \| \overline{u} - v_h \|^2_{L^2(Q)} \Big] \\
    & & \hspace*{-3cm} \leq \, 8 \, \varrho \, \| \overline{u} \|_X^2 + 2
             \inf\limits_{v_h \in X_h} \Big[
             2 \, \varrho \, \| \overline{u} - v_h \|^2_X +
             \| \overline{u} - v_h \|^2_{L^2(Q)} \Big] \\
    & & \hspace*{-3cm} \leq \, 8 \, \max \{ c_F, 1 \} \,
        \varrho \, \| \overline{u} \|_{H^1(Q)}^2 \\
    && \hspace*{-1cm} + 2
             \inf\limits_{v_h \in X_h} \Big[
        2 \, \max \{ c_F , 1 \} \,
        \varrho \, \| \overline{u} - v_h \|^2_{H^1(Q)} +
             \| \overline{u} - v_h \|^2_{L^2(Q)} \Big]
  \end{eqnarray*}
  when using \eqref{Abstract error X X} and \eqref{Abstract error H X},
  the upper norm equivalence inequality in \eqref{norm equivalence}
  with $c_2^S=2$, and $\| \overline{u} \|_{H^1(Q)}$ as upper bound
  of $\| \overline{u} \|_X$, see \eqref{Norm X H1}.
  Now inserting a suitable $H^1$-stable quasi-interpolation
  $v_h = P_h \overline{u} \in X_h$ 
  of the desired state $\overline{u} \in H^1(Q)$, e.g.,
  Scott--Zhang's interpolation \cite{LSY:BrennerScott:2008a},
  we immediately obtain the estimate
  \[
    \| u_\varrho - u_{\varrho h} \|_{L^2(Q)}^2 \le
    c \, [\varrho + h^2] \, \| \overline{u} \|^2_{H^1(Q)}.
  \]
  Combining this estimate with \eqref{Abstract error H X} and
  chosing $\varrho = h^2$ finally gives
  \[
    \| u_{\varrho h} - \overline{u} \|_{L^2(Q)} \leq c \, h \,
    \| \overline{u} \|_{H^1(Q)} .
  \]
  Next we consider $\overline{u} \in X \cap H^2(Q)$ which guarantees
  $S \overline{u} \in L^2(Q)$.
  Similar as above, but now using \eqref{Abstract error H Su} and
  \eqref{Abstract error S Su}, we then obtain the estimates
  \begin{eqnarray*}
    \| u_{\varrho h} - \overline{u}\|_{L^2(Q)}^2
    & \leq & 2 \, \|  u_{\varrho h} - u_\varrho  \|_{L^2(Q)}^2 +
             2 \, \| u_\varrho - \overline{u} \|_{L^2(Q)}^2 \\
    & \leq & 10 \, \varrho^2 \, \| S \overline{u} \|^2_{L^2(Q)}
             + 4 \inf\limits_{v_h \in X_h}
             \Big[ \varrho \, \| \overline{u} - v_h \|^2_S +
             \| \overline{u} - v_h \|^2_{L^2(Q)} \Big] \\
    & \leq & c \,[\varrho^2 + \varrho h^2 + h^4] \,
             \| \overline{u} \|_{H^2(Q)} \, .
  \end{eqnarray*}
  Here we have used the estimate
  \[
    \| S \overline{u} \|_{L^2(Q)} \leq c \, \| \overline{u} \|_{H^2(Q)}
  \]
  that can be shown by Fourier analysis; cf.
  \cite{LSY:SteinbachZank:2020a}. Chosing $\varrho = h^2$ yields
  \[
    \| u_{\varrho h} - \overline{u} \|_{L^2(Q)} \leq c \, h^2 \,
    \| \overline{u}\|_{H^2(Q)} \, .
  \]
  The general estimate for $s \in [0,1)$ and $s \in [1,2]$ 
  now follows from a 
  space interpolation argument; see, e.g., \cite{LSY:Tartar:2007a}.
\end{proof}

\begin{cor}
  Let us assume that   $\overline{u} \in X \cap H^s(Q)$
  for some $s \in [1,2]$. Then
  there holds the error estimate
  \begin{equation}\label{Error H1}
    \| u_{\varrho h} - \overline{u} \|_X \leq c \,
    h^{s-1} \|\overline{u}\|_{H^s(Q)} .
  \end{equation}
\end{cor}

\begin{proof}
  Let $P_h\overline{u} \in X_h$ be again Scott--Zhang's
  interpolation of $\overline{u} \in H^1(Q)$.
  Using an inverse inequality and standard
  arguments we obtain
  \begin{eqnarray*}
    \| u_{\varrho h} - \overline{u} \|_X
    & \leq & \| u_{\varrho h} - \overline{u} \|_{H^1(Q)} \\
    & \leq & \| u_{\varrho h} - P_h \overline{u} \|_{H^1(Q)} +
             \| P_h \overline{u} - \overline{u} \|_{H^1(Q)} \\
    & \leq & c \, h^{-1} \, \| u_{\varrho h} - P_h \overline{u} \|_{L^2(Q)} +
             c \, h^{s-1} \, \| \overline{u} \|_{H^s(Q)} \\
    & \leq & c \, h^{-1} \, \Big[
             \| u_{\varrho h} - \overline{u} \|_{L^2(Q)} +
             \| \overline{u} - P_h \overline{u} \|_{L^2(Q)} \Big] +
             c \, h^{s-1} \, \| \overline{u} \|_{H^s(Q)} \\
    & \leq & c \, h^{s-1} \,
    \| \overline{u} \|_{H^s(Q)} \, .
  \end{eqnarray*} 
\end{proof}

\noindent
Since \eqref{VF parabolic S FEM} requires, for any given $ w \in X$,
the evaluation of $Sw  = B^* A^{-1} B w$, we have to define
a suitable computable
approximation $\widetilde{S}w$. This can be done as follows.
For given $ w \in X$, we introduce $ p_w = A^{-1} B w \in Y$ 
as the unique solution of the variational formulation
\[
  \langle A p_w , q \rangle_Q = \langle B w , q \rangle_Q \quad
  \forall \, q \in Y .
\]
Let $p_{wh} \in Y_h  \subset Y$ be the continuous, piecewise linear
space-time finite element approximation to $p_w \in Y$, satisfying
\begin{equation}
\label{eqn:pwh}
 \langle A p_{wh} , q_h \rangle_Q = \langle B w , q_h \rangle_Q \quad
  \forall \, q_h \in Y_h .
\end{equation} 
With this we define the approximate operator $\widetilde{S}w := B^* p_{wh}$
of $Sw = B^* p_w$. The boundedness of $B : X \to Y^*$ implies
\[
  \| \widetilde{S} w \|_{X^*} = \| B^* p_{wh} \|_{X^*} \leq c_2^B \,
  \| p_{wh} \|_X,
\]
while the ellipticity of $A : Y \to Y^*$ gives
\[
  c_1^A \, \| p_{wh} \|^2_Y \leq \langle A p_{wh} , p_{wh} \rangle_Q =
  \langle B w , p_{wh} \rangle_Q \leq c_2^B \, \| w \|_X \| p_{wh} \|_Y ,
\]
i.e.,
\[
\| p_{wh} \|_Y \leq \frac{c_2^B}{c_1^A} \, \| w \|_X \, .
\]
Hence, we conclude the boundedness of the approximate operator
$\widetilde{S} : X \to X^*$,
\begin{equation}\label{bound Stilde}
  \| \widetilde{S} w \|_{X^*} \leq c_2^{\widetilde{S}} \, \| w \|_X \quad
  \forall \, w \in X, \quad c_2^{\widetilde{S}} =
  \frac{[c_2^B]^2}{c_1^A} = 2 .
\end{equation}
Instead of \eqref{VF parabolic S FEM}, we now
consider the perturbed variational formulation to find
$\widetilde{u}_{\varrho h} \in X_h$ such that
\begin{equation}\label{VF parabolic S FEM pert}
  \varrho \, \langle \widetilde{S} \widetilde{u}_{\varrho h} , v_h \rangle_Q +
  \langle \widetilde{u}_{\varrho h} , v_h \rangle_{L^2(Q)} =
  \langle \overline{u} , v_h \rangle_{L^2(Q)}
\end{equation}
is satisfied for all $v_h \in X_h$. Unique solvability of
\eqref{VF parabolic S FEM pert} follows since the stiffness matrix
of $\widetilde{S}$ is positive semi-definite, while the mass
matrix, which is related to the inner product in $L^2(Q)$,
is positive definite.

\begin{lemma}\label{Lemma Error rho h}
  Let $u_{\varrho h} \in X_h$ and $\widetilde{u}_{\varrho h} \in X_h$ be
  the unique solutions of the variational formulations
  \eqref{VF parabolic S FEM} and \eqref{VF parabolic S FEM pert},
  respectively.  Assume $\overline{u} \in X \cap H^1(Q)$. Then,
  there holds the error estimate
  \[
    \| u_{\varrho h} - \widetilde{u}_{\varrho h} \|_{L^2(Q)} \leq
    c \, h \, \| \overline{u} \|_{H^1(Q)} .
  \]
\end{lemma}
\begin{proof}
  The difference of the variational formulations
  \eqref{VF parabolic S FEM} and \eqref{VF parabolic S FEM pert}
  first gives the Galerkin orthogonality
  \[
    \varrho \, \langle S u_{\varrho h} -
    \widetilde{S} \widetilde{u}_{\varrho h} , v_h \rangle_Q +
    \langle u_{\varrho h} - \widetilde{u}_{\varrho h} , v_h \rangle_{L^2(Q)} = 0
    \quad \forall \, v_h \in X_h ,
  \]
  which can be written as
  \[
    \varrho \, \langle \widetilde{S}
    (\widetilde{u}_{\varrho h}-u_{\varrho h}) , v_h \rangle_Q
    +
    \langle \widetilde{u}_{\varrho h} - u_{\varrho h} , v_h \rangle_{L^2(Q)} 
    =
    \varrho \, \langle (S - \widetilde{S}) u_{\varrho h} , v_h \rangle_Q 
    \quad \forall \, v_h \in X_h .
  \]
  In particular, chosing $v_h = \widetilde{u}_{\varrho h} - u_{\varrho h} \in X_h$,
 using $\langle \widetilde{S} w , w \rangle_Q \geq 0$ for all $w \in X$,
 applying an inverse inequality in $X_h$, 
 i.e., using the dual norm for $\| \partial_t v_h \|_{Y^*}$  
 and Friedrich's inequality \eqref{FriedrichsInequality},
 we arrive at the estimates
  \begin{eqnarray*}
    \| \widetilde{u}_{\varrho h} - u_{\varrho h} \|^2_{L^2(Q)}
    & \leq & \varrho \, \langle (S - \widetilde{S}) u_{\varrho h} ,
             \widetilde{u}_{\varrho h} - u_{\varrho h}  \rangle_Q \\
    & \leq & \varrho \, \| (S-\widetilde{S}) u_{\varrho h} \|_{X^*}
             \| \widetilde{u}_{\varrho h} - u_{\varrho h} \|_X \\
    & \leq & c \, \varrho \, h^{-1} \,
             \| (S-\widetilde{S}) u_{\varrho h} \|_{X^*}
            \| \widetilde{u}_{\varrho h} - u_{\varrho h} \|_{L^2(Q)},
  \end{eqnarray*}
  i.e.,
  \[
    \| \widetilde{u}_{\varrho h} - u_{\varrho h} \|_{L^2(Q)} \leq
    c \, \varrho \, h^{-1} \, \| (S-\widetilde{S}) u_{\varrho h} \|_{X^*} .
  \]
  Since $\overline{u} \in X$, we can further estimate
  \begin{eqnarray*}
    \| \widetilde{u}_{\varrho h} - u_{\varrho h} \|_{L^2(Q)}
    & \leq & c \, \varrho \, h^{-1} \, \Big[
             \| (S-\widetilde{S}) (u_{\varrho h}-\overline{u}) \|_{X^*} +
             \| (S-\widetilde{S}) \overline{u} \|_{X^*} \Big] \\
    & \leq & c \, \varrho \, h^{-1} \, \Big[
             4 \, \| u_{\varrho h} - \overline{u} \|_X + \sqrt{2} \,
             \| p_{\overline{u}} - p_{\overline{u} h} \|_Y \Big] \, ,
  \end{eqnarray*}
  where we used the boundedness of $S$ and $\widetilde{S}$.
  We note that $p_{\overline{u}} = A^{-1} B \overline{u}$, 
  and $p_{\overline{u} h} \in Y_h$ solves \eqref{eqn:pwh} with
  $w = \overline{u}$.
  For $\overline{u} \in H^1(Q)$, we can use standard arguments as well as
  \eqref{Norm X H1} to bound
  \[
    \| p_{\overline{u}} - p_{\overline{u} h} \|_Y \leq
    \| p_{\overline{u}} \|_Y =
    \| A^{-1} B \overline{u} \|_Y \leq \frac{c_2^B}{c_1^A} \,
    \| \overline{u} \|_X \leq \sqrt{2} \,
    \max \{ \sqrt{c_F} , 1 \} \, \| \overline{u} \|_{H^1(Q)} ,
  \]
  and using \eqref{Error H1} for $s=1$ we finally obtain, using
  $\varrho = h^2$,
  \[
    \| \widetilde{u}_{\varrho h} - u_{\varrho h} \|_{L^2(Q)} \, \leq \, 
    c \, h \, \| \overline{u} \|_{H^1(Q)} .
  \]
\end{proof}

\begin{theorem}
  Assume $\overline{u} \in [L^2(Q),X]_s \cap H^s(Q)$ for $s \in [0,1]$,
  and $\varrho = h^2$. Then,
  \begin{equation}\label{final estimate 0 1}
    \| \widetilde{u}_{\varrho h} - \overline{u} \|_{L^2(Q)} \leq
    c \, h^s \, \|\overline{u} \|_{H^s(Q)} .
  \end{equation}
 \end{theorem} 

  \begin{proof}
    For $s = 1$ the assertion is an immediate consequence of
    Theorem \ref{Theorem1} and Lemma \ref{Lemma Error rho h}.
    Now we consider \eqref{VF parabolic S FEM pert} for
  $v_h = \widetilde{u}_{\varrho h} $,
  \[
    \varrho \, \langle \widetilde{S} \widetilde{u}_{\varrho h} ,
    \widetilde{u}_{\varrho h} \rangle_Q +
    \langle \widetilde{u}_{\varrho h} - \overline{u} ,
    \widetilde{u}_{\varrho h} - \overline{u} \rangle_{L^2(Q)} = 
    \langle \overline{u} - \widetilde{u}_{\varrho h} ,
    \overline{u} \rangle_{L^2(Q)} ,
  \]
  from which we immediately conclude
  \[
    \| \widetilde{u}_{\varrho h} - \overline{u} \|_{L^2(Q)} \leq
    \| \overline{u} \|_{L^2(Q)} .
  \]
  The assertion then again follows by a space interpolation argument.
\end{proof}

\noindent
The error estimate as given in \eqref{final estimate 0 1} covers
in particular the case when the target is either discontinuous, or
does not satisfy the required boundary or initial conditions.
It remains to consider the case when the target $\overline{u}$ is
smooth. As in the proof of Lemma \ref{Lemma Error rho h}, and 
using \eqref{Error H1} for $s=2$, we now have, recall $\varrho = h^2$,
\begin{eqnarray*}
  \| \widetilde{u}_{\varrho h} - u_{\varrho h} \|_{L^2(Q)}
  & \leq & c \, \varrho \, h^{-1} \, \Big[
           4 \, \| u_{\varrho h} - \overline{u} \|_X + \sqrt{2} \,
           \| p_{\overline{u}} - p_{\overline{u} h} \|_Y \Big] \\
  & \leq & c_1 \, h^2 \, \| \overline{u} \|_{H^2(Q)} + c_2 \, h \,
           \| p_{\overline{u}} - p_{\overline{u} h} \|_Y \, .
\end{eqnarray*}
When using the approximation result as given in
\cite[Theorem 3.3]{LSY:Steinbach:2015a} we have
\begin{equation}\label{space time FEM error Y}
  \| p_{\overline{u}} - p_{\overline{u} h} \|_Y \leq c \, h \,
  \| p_{\overline{u}} \|_{H^2(Q)},
\end{equation}
i.e., we obtain
\begin{equation}\label{final estimate 2 p}
 \| \widetilde{u}_{\varrho h} - u_{\varrho h} \|_{L^2(Q)}
 \leq c \, h^2 \, \Big[ \| \overline{u} \|_{H^2(Q)} +
 \| p_{\overline{u}} \|_{H^2(Q)} \Big] .
\end{equation}
While the error estimate \eqref{space time FEM error Y} holds for any
admissible decomposition of the space-time domain $Q$ into simplicial
finite elements, in addition to $\overline{u} \in X \cap H^2(Q)$, we have
to assume $p_{\overline{u}} = A^{-1} B \overline{u} \in H^2(Q)$,
i.e., $\overline{u} \in H^{2,3}(Q)$. This additional regularity requirement
in time is due to the finite element error estimate
\eqref{space time FEM error Y} which does not reflect the anisotropic
behavior in space and time of the norm in $Y=L^2(0,T;H^1_0(\Omega))$.
However, and as already discussed in \cite[Corollary 4.2]{LSY:Steinbach:2015a},
we can improve the error estimate \eqref{space time FEM error Y} under
additional assumptions on the underlying space-time finite element mesh.
In fact, when considering as in \cite[Section 4]{LSY:Steinbach:2015a}
right-angled space-time finite elements, or space-time tensor product
meshes, instead of \eqref{space time FEM error Y} we obtain the error
estimate
\begin{equation}\label{space time FEM error Y improved}
  \| p_{\overline{u}} - p_{\overline{u} h} \|_Y \leq c \, h \,
  | \nabla_x p_{\overline{u}} |_{H^1(Q)},
\end{equation}
when assuming $\nabla_x p_{\overline{u}} \in H^1(Q)$ for
$p_{\overline{u}} = A^{-1} B \overline{u} $,, i.e., there are no
second order time derivatives yet. This is the reason to further conclude
the bound
\[
  \| p_{\overline{u}} - p_{\overline{u} h} \|_Y \leq c \, h \,
  \| \overline{u} \|_{H^2(Q)},
\]
and hence,
\begin{equation}\label{final estimate 2 p improved}
 \| \widetilde{u}_{\varrho h} - u_{\varrho h} \|_{L^2(Q)}
 \leq c \, h^2 \, \| \overline{u} \|_{H^2(Q)} 
\end{equation}
follows, when assuming $\overline{u} \in X \cap H^2(Q)$. Now,
interpolating \eqref{final estimate 0 1} for $s=1$ and
\eqref{final estimate 2 p improved}, we conclude
\begin{equation}\label{final estimate 1 2}
 \| \widetilde{u}_{\varrho h} - u_{\varrho h} \|_{L^2(Q)}
 \leq c \, h^s \, \| \overline{u} \|_{H^s(Q)} \quad
 \mbox{for} \; \overline{u} \in X \cap H^s(Q), \; s \in [1,2]
\end{equation}
that together with estimate \eqref{Th1Estimate} from Theorem~\ref{Theorem1}
finally gives
\begin{equation}\label{FinalEstimate0s}
  \| \widetilde{u}_{\varrho h} - \overline{u} \|_{L^2(Q)}
  \leq c \, h^s \, \| \overline{u} \|_{H^s(Q)}.
\end{equation}
While we can prove this result for some structured space-time finite element
meshes only, numerical experiments indicate that \eqref{FinalEstimate0s}
remains true for any admissible decomposition
of the space-time domain into simplicial finite elements.

\section{Numerical results}
\label{sec:NumericalResults}
In the numerical experiments, we choose the spatial domain $\Omega=(0,1)^n$
with $n=2$ (Subsection~\ref{subsec:2d}) and $n=3$
(Subsection~\ref{subsec:3d}), and final time $T=1$,
resulting in the $n+1$-dimensional space-time cylinder $Q = (0,1)^{n+1}$. 
We follow the space-time finite element method on fully 
unstructured simplicial meshes as considered in
\cite{LSY:LangerSteinbachTroeltzschYang:2021c} for the coupled
optimality system of the parabolic distributed control problem
\eqref{Parabolic minimization problem}. This finally leads to the solution
of a saddle-point system that is nothing but the discrete version of
\eqref{rOS}:
Find the nodal parameter vectors $\underline{p} \in \mathbb{R}^{M_Y}$
($M_Y = \text{dim}(Y_h)$) 
and $\underline{u} \in \mathbb{R}^{M_X}$ ($M_X = \text{dim}(X_h)$) 
such that
\begin{equation}
\label{saddle-point system} 
\begin{pmatrix} \varrho^{-1} A_h & B_h \\ 
                B_h^\top & - M_h 
\end{pmatrix}
\begin{pmatrix} \underline{p}  \\  \underline{u} \end{pmatrix} =
\begin{pmatrix} \underline{0} \\ -\underline{f} \end{pmatrix},
\end{equation}
where the finite element matrices $A_h$, $B_h$, and  $M_h$  
correspond to the bilinear forms \eqref{Definition_A} and
\eqref{Definition_B}, and to the $L_2(Q)$ inner product, respectively.
The matrices $A_h \in {\mathbb{R}}^{M_Y \times M_Y}$ and
$M_h \in {\mathbb{R}}^{M_X \times M_X}$ are symmetric and positive definite,
while the matrix $B_h \in {\mathbb{R}}^{M_Y \times M_X}$ is in general
rectangular.
The load vector $\underline{f} \in {\mathbb{R}}^{N_X}$ is computed from
the given target $\overline{u}$ as usual. We mention that the symmetric,
but indefinite system \eqref{saddle-point system} is equivalent to solving the 
related Schur complement system  
\[
  \left( \varrho B_h^\top A_h^{-1} B_h + M_h \right)
  \underline{u} = \underline{f}
\]
that corresponds to \eqref{VF parabolic S FEM pert}. Here, the symmetric
but indefinite system \eqref{saddle-point system} is simply solved by
the ILU(0) preconditioned GMRES method; 
see \cite{LSY:LangerSteinbachTroeltzschYang:2021c}. 
We stop the GMRES iteration when the
relative residual error of the preconditioned system is reduced by a factor
$10^8$. 

%
%
\subsection{Two space dimensions}
\label{subsec:2d}
In the first example (Example~4.1.1), we consider the smooth target
\begin{equation}\label{Example 1}
  \overline{u}(x,t)=\sin(\pi x_1)\sin(\pi x_2)\sin(\pi t)
\end{equation}
where we can apply the error estimate \eqref{final estimate 2 p}.
As predicted, we observe a second order convergence with respect to the
mesh size $h$ when choosing $\varrho=h^2$; see Table~\ref{tab:ex2smoothtar}.

\begin{table}
  \centering
  \begin{tabular}{|l|l|ll|}
    \hline
    $h$ & $\varrho \,(=h^2)$&
    $\|\widetilde{u}_{\varrho h}-\overline{u}\|_{L^2(Q)}$ & eoc  \\ \hline
    $2^{-2}$& $2^{-4}$  & $2.2380$e$-1$& \\ 
    $2^{-3}$& $2^{-6}$  & $9.0449$e$-2$&$1.31$ \\ 
    $2^{-4}$& $2^{-8}$  & $2.6491$e$-2$&$1.77$ \\ 
    $2^{-5}$& $2^{-10}$ & $6.9335$e$-3$&$1.93$\\ 
    $2^{-6}$& $2^{-12}$ & $1.7613$e$-3$&$1.98$\\
    $2^{-7}$& $2^{-14}$& $4.4352$e$-4$&$1.99$\\    
    $2^{-8}$& $2^{-16}$ & $1.0600$e$-4$&$2.06$\\       
    $2^{-9}$& $2^{-18}$ & $2.6836$e$-5$&$1.98$\\       
    \hline
  \end{tabular}
  \caption{Error $\|\widetilde{u}_{\varrho h}-\overline{u}\|_{L^2(Q)}$
    in the case of a smooth target $\overline{u}$ given by
    \eqref{Example 1} (Example~4.1.1).}  
  \label{tab:ex2smoothtar}
\end{table}

As a second example (Example~4.1.2), we consider a piecewise linear continuous
function $\overline{u}$ being one at the mid point
$(1/2,1/2,1/2)$, and zero in all corner points of $Q=(0,1)^3$.
In this case, we have $\overline{u} \in X \cap H^{3/2-\varepsilon}(Q)$,
$\varepsilon >0$, and we observe $1.5$ as
the order of convergence, see Table~\ref{tab:ex2contar}, which
corresponds to the error estimate \eqref{final estimate 1 2}.

\begin{table}
  \centering
  \begin{tabular}{|l|l|ll|}
    \hline
    $h$ & $\varrho \,(=h^2)$& $\|\widetilde{u}_{\varrho h}-\overline{u}\|_{L^2(Q)}$ & eoc  \\ \hline
    $2^{-2}$& $2^{-4}$  & $2.0231$e$-1$& \\ 
    $2^{-3}$& $2^{-6}$  & $9.1319$e$-2$&$1.15$\\ 
    $2^{-4}$& $2^{-8}$  & $3.4303$e$-2$&$1.41$\\ 
    $2^{-5}$& $2^{-10}$ & $1.2428$e$-2$&$1.46$\\ 
    $2^{-6}$& $2^{-12}$ & $4.4443$e$-3$&$1.48$\\
    $2^{-7}$& $2^{-14}$ & $1.5797$e$-3$&$1.49$\\
    $2^{-8}$& $2^{-16}$ & $5.5868$e$-4$&$1.50$\\       
    $2^{-9}$& $2^{-18}$ & $1.9786$e$-4$&$1.50$\\       
    \hline
  \end{tabular}
  \caption{Error $\|\widetilde{u}_{\varrho h}-\overline{u}\|_{L^2(Q)}$ in
    the case of a piecewise
    linear continuous target $\overline{u} \in X \cap H^{3/2-\varepsilon}(Q)$,
    $\varepsilon > 0$ (Example~4.1.2).}  
  \label{tab:ex2contar}
\end{table}

As a third example (Example~4.1.3), we take
a piecewise constant discontinuous
target $\overline{u}$ which is one in the inscribed cube
$(\frac{1}{4},\frac{3}{4})^3$, and zero elsewhere. 
In this case, we have
$\overline{u} \in H^{1/2-\varepsilon}(Q)$, $\varepsilon >0$.
From the numerical results given in Table \ref{tab:ex2tardiscon},
we observe $0.5$ for the order of convergence, as expected from the
error estimate \eqref{final estimate 0 1}. In this example, since the
target $\overline{u}$ is discontinuous, we may apply an adaptive refinement
based on the residual type error indicator as used in
\cite{LSY:LangerSteinbachTroeltzschYang:2021c}.
We compare the errors and number of degrees of freedom using both uniform and
adaptive refinements in Table \ref{tab:ex2tardisconadapt}, with respect to the
regularization parameter $\varrho$. We clearly see that for each
regularization parameter $\varrho$, 
the adaptive refinement requires less degrees of freedom to reach a similar
accuracy as for uniform refinements. 
In Figure~\ref{fig:ex2distarvis}, we plot the state $u$, the adjoint state
$p$, and the control $z$ at time $t=0.5$, and the adaptive meshes in
space-time. For comparison of the results with different regularization
terms, we refer to
the numerical results in \cite{LSY:LangerSteinbachTroeltzschYang:2021c}.

\begin{table}
  \centering
  \begin{tabular}{|l|l|ll|}
    \hline
    $h$ & $\varrho \,(=h^2)$& $\|\widetilde{u}_{\varrho h}-\overline{u}\|_{L^2(Q)}$ & eoc  \\ \hline
    $2^{-2}$& $2^{-4}$  & $2.8840$e$-1$& \\ 
    $2^{-3}$& $2^{-6}$  & $2.0871$e$-1$&$0.47$\\ 
    $2^{-4}$& $2^{-8}$  & $1.4793$e$-1$&$0.50$\\
    $2^{-5}$& $2^{-10}$ & $1.0473$e$-1$&$0.50$\\ 
    $2^{-6}$& $2^{-12}$ & $7.4108$e$-2$&$0.50$\\
    $2^{-7}$& $2^{-14}$ & $5.2425$e$-2$&$0.50$\\
    $2^{-8}$& $2^{-16}$ & $3.7079$e$-2$&$0.50$\\       
    $2^{-9}$& $2^{-18}$ & $2.6219$e$-2$&$0.50$\\       
    \hline
  \end{tabular}
  \caption{Error $\|\widetilde{u}_{\varrho h}-\overline{u}\|_{L^2(Q)}$
    in the case of a
    discontinuous target $\overline{u}\in H^{1/2-\varepsilon}(Q)$,
    $\varepsilon>0$ (EXample~4.1.3).}   
  \label{tab:ex2tardiscon}
\end{table}

\begin{table}
  \centering
  \begin{tabular}{|l|lll|ll|}
    \hline
    & \multicolumn{3}{l|}{uniform refinement}&\multicolumn{2}{l|}{adaptive
                                                 refinement} \\
    \hline
    $\varrho $& 
    $h = \varrho^{1/2}$
    & \#DOFs & $\|\widetilde{u}_{\varrho h}-\overline{u}\|_{L^2(Q)}$&\#DOFs&$\|\widetilde{u}_{\varrho
                                                                                     h}-\overline{u}\|_{L^2(Q)}$
    \\ \hline
    $2^{-4}$& $2^{-2}$  &$250$& $2.8840$e$-1$&$250$&$2.8840$e$-1$ \\ 
    $2^{-6}$& $2^{-3}$  &$1,458$& $2.0871$e$-1$&$1,230$&$2.0873$e$-1$\\ 
    $2^{-8}$& $2^{-4}$  &$9,826$& $1.4793$e$-1$&$9,948$&$1.3999$e$-1$\\
    $2^{-10}$& $2^{-5}$ &$71,874$& $1.0473$e$-1$&$34,998$&$1.0153$e$-1$\\ 
    $2^{-12}$& $2^{-6}$ &$549,250$& $7.4108$e$-2$&$230,154$&$7.2804$e$-2$\\
    $2^{-14}$& $2^{-7}$ &$4,293,378$& $5.2425$e$-2$&$1,526,400$&$5.1838$e$-2$\\
    $2^{-16}$& $2^{-8}$ &$33,949,186$& $3.7079$e$-2$&$6,196,200$&$3.6609$e$-2$\\       
    $2^{-18}$& $2^{-9}$ &$270,011,394$& $2.6219$e$-2$&$31,419,720$&$2.5824$e$-2$\\       
    \hline
  \end{tabular}
  \caption{Comparison of the error 
     $\|\widetilde{u}_{\varrho h}-\overline{u}\|_{L^2(Q)}$ 
     and the number of degrees of freedoms 
     in the case of a discontinuous target
     $\overline{u}\in H^{1/2-\varepsilon}(Q)$, $\varepsilon >0$,
    when using both uniform and adaptive refinements (Example~4.1.3).}   
  \label{tab:ex2tardisconadapt}
\end{table}

  \begin{figure}
  \centering
  \includegraphics[width=0.45\textwidth]{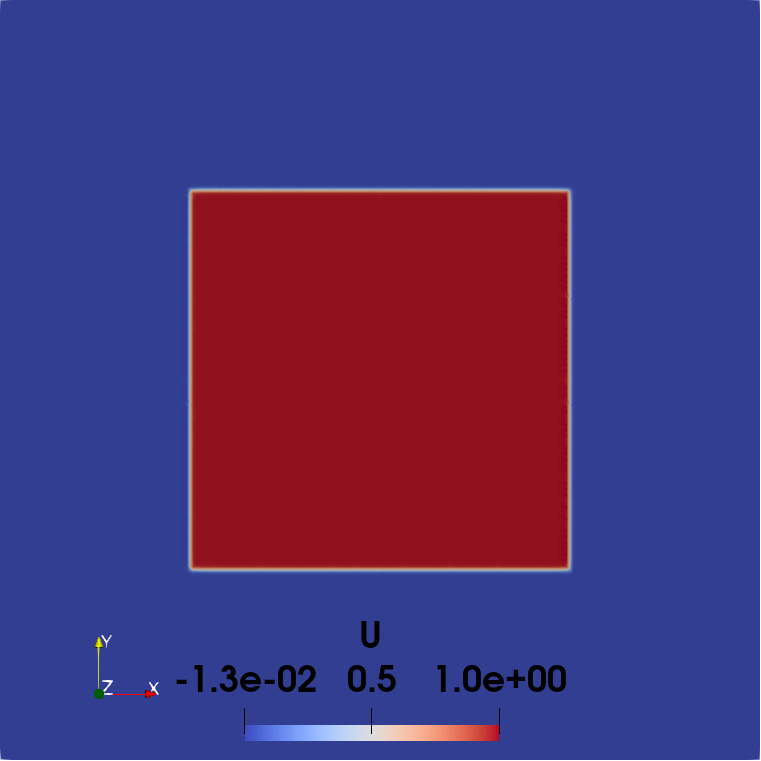}
  \includegraphics[width=0.45\textwidth]{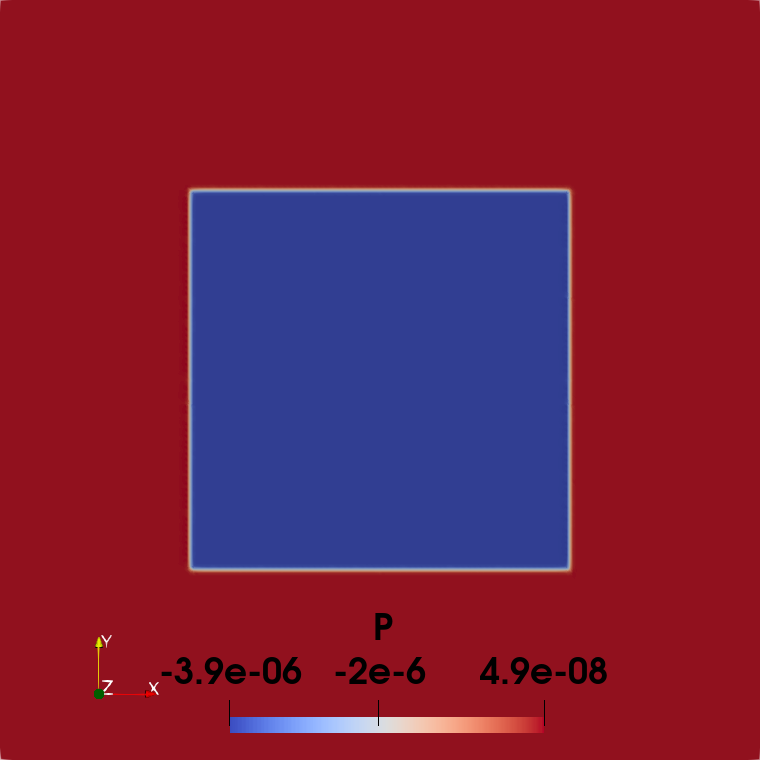}
  \includegraphics[width=0.45\textwidth]{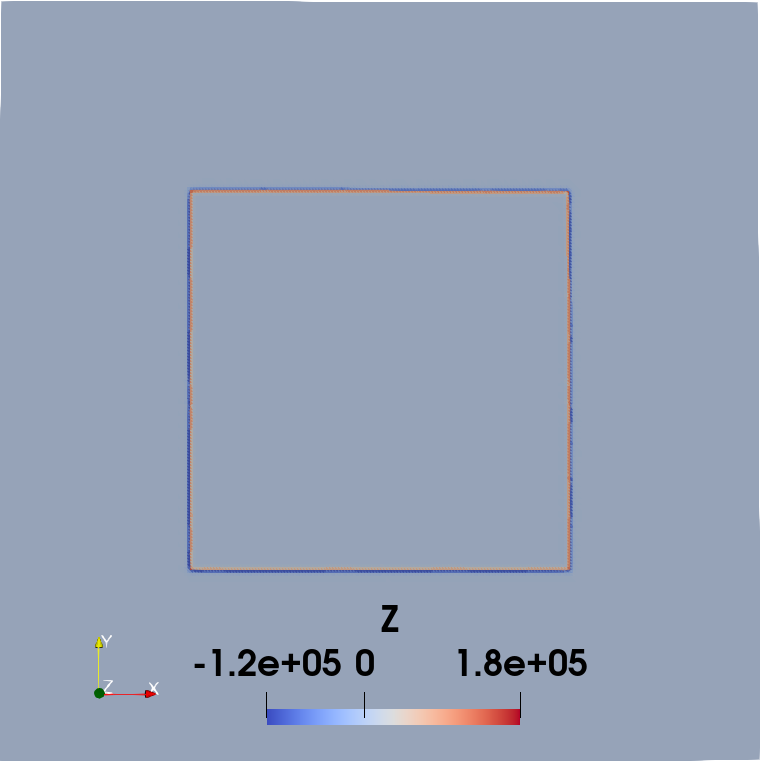}
  \includegraphics[width=0.45\textwidth]{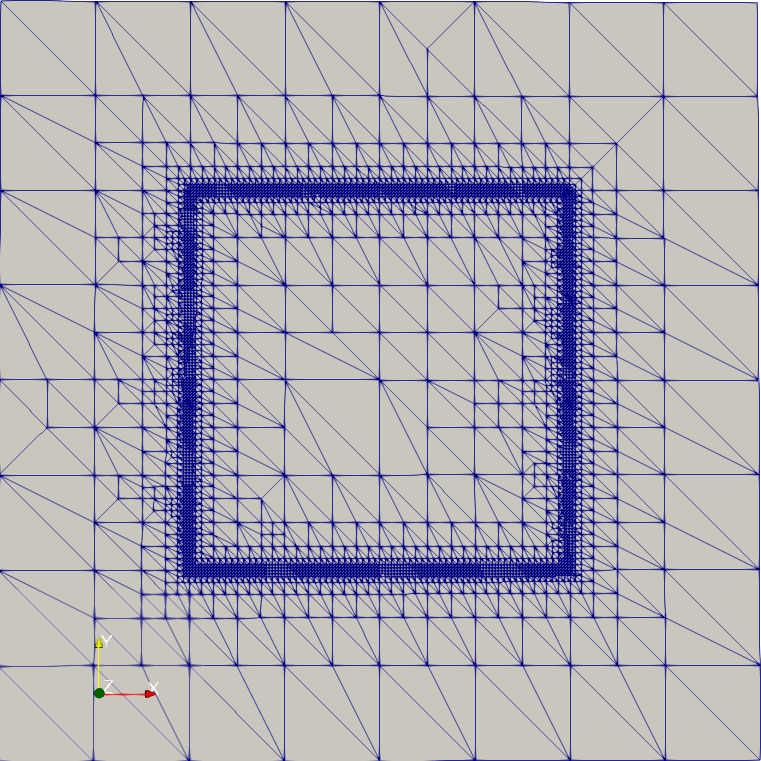}
  \caption{Visualization of the state $u_\varrho$, the adjoint state
    $p_\varrho$, the control $z_\varrho$, and
    the adaptive mesh on the cutting plane at time $t=0.5$, where the total
    \#DOFs in space-time is $3,328,617$ at the $58$th adaptive level  
    from $67$ levels (last line in Table~\ref{tab:ex2tardisconadapt})
    corresponding to the regularization parameter $\varrho=2^{-18}$
    (Example~4.1.3).}   
  \label{fig:ex2distarvis}
\end{figure} 

In the last example (Example~4.1.4) of this subsection,
we consider the discontinuous target 
\begin{equation}\label{Example noise}
  \overline{u}^\delta =
  \overline{u}+2\sqrt{2}\delta\sin(10\pi x_1)\sin(10\pi x_2)\sin(10\pi t),
\end{equation}
that contains some noise in space and time. Here, $\overline{u}$
is one in the inscribed cube $(1/4,3/4)^3 \subset (0,1)^3$ and zero else,
and $\delta > 0$ is the noise level. For this polluted target, we easily
see that $\|\overline{u}^\delta-\overline{u}\|_{L^2(Q)}=\delta$. To balance
the two error contributions we take $h=16 \, \delta^2$. This ensures an
almost optimal convergence with respect to the mesh size $h$, see
Table \ref{tab:ex2tardisconnoise1}.

\begin{table}
  \centering
  \begin{tabular}{|l|l|l|ll|}
      \hline
      $\delta$ & $h\, \left(=16\cdot\delta^2\right)$ &  $\varrho
                                                       \,\left(=h^2\right)$&
                                                                             $\|\widetilde{u}_{\varrho h}^\delta-\overline{u}\|_{L^2(Q)}$ & eoc  \\ \hline
      $2^{-3}$&$2^{-2}$ & $2^{-4}$  & $2.8841$e$-1$& \\
      $2^{-3.5}$&$2^{-3}$ &$2^{-6}$  & $2.0871$e$-1$&$0.47$ \\
      $2^{-4}$&$2^{-4}$&$2^{-8}$  & $1.4796$e$-1$&$0.50$ \\
      $2^{-4.5}$&$2^{-5}$&$2^{-10}$ & $1.0535$e$-1$&$0.49$\\
      $2^{-5}$&$2^{-6}$&$2^{-12}$ & $7.6837$e$-2$&$0.46$\\
      $2^{-5.5}$&$2^{-7}$ &$2^{-14}$ & $5.5990$e$-2$&$0.46$\\
      \hline
  \end{tabular}
  \caption{Error $\|\widetilde{u}_{\varrho h}^\delta-\overline{u}\|_{L^2(Q)}$
    in the case of a
    discontinuous target $\overline{u}\in H^{1/2-\varepsilon}(Q)$
    containing some noise level $\delta$ (Example~4.1.4).}  
  \label{tab:ex2tardisconnoise1}
\end{table}

%
%
\subsection{Three space dimensions}
\label{subsec:3d}
Now we present some numerical results for the three-dimensional
spatial domain $\Omega = (0,1)^3$, i.e., $Q=(0,1)^4$. 

In the first example (Example~4.2.1), we look at
the smooth target
\begin{equation}\label{Example 3D 1}
  \overline{u}(x,t) =
  \sin(\pi x_1)\sin(\pi x_2)\sin(\pi x_3)\sin(\pi t) .
\end{equation}
As predicted by the error estimate \eqref{final estimate 2 p},
we observe a second order convergence with respect to the
mesh size $h$ when choosing $\varrho=h^2$; see Table~\ref{tab:ex2smoothtar4D}
and Figure~\ref{fig:ex2smoothtar4D}.

\begin{table}
  \centering
  \begin{tabular}{|l|l|l|l|}
    \hline
    N(\#DOFs)&$h=(N/2)^{-1/4}$ & $\varrho \,(=h^2)$& $\|\widetilde{u}_{\varrho h}-\overline{u}\|_{L^2(Q)}$  \\ \hline
    $356$&$2.7378$e$-1$& $7.4953$e$-2$  & $2.0985$e$-1$ \\
    $630$&$2.3737$e$-1$& $5.6344$e$-2$  & $1.7263$e$-1$ \\
    $2,986$&$1.6087$e$-1$& $2.5880$e$-2$ & $1.2166$e$-1$\\
    $6,930$&$1.3034$e$-1$& $1.6988$e$-2$ & $9.3733$e$-2$\\
    $38,114$&$8.5111$e$-2$& $7.2439$e$-3$ & $4.7198$e$-2$\\
    $94,146$&$6.7890$e$-2$& $4.6091$e$-3$ & $3.2098$e$-2$\\
    $546,562$&$4.3737$e$-2$& $1.9129$e$-3$ & $1.4088$e$-2$\\
    $1,400,322$&$3.4570$e$-2$& $1.1951$e$-3$ & $8.8652$e$-3$\\
    $8,289,026$&$2.2163$e$-2$& $4.9121$e$-4$ & $3.7164$e$-3$\\
    $21,657,090$&$1.7432$e$-2$& $3.0389$e$-4$ & $2.2967$e$-3$\\
    $129,165,826$&$1.1155$e$-2$& $1.2444$e$-4$ & $9.5061$e$-4$\\
    \hline
  \end{tabular}
  \caption{Error $\| \widetilde{u}_{\varrho h}-\overline{u}\|_{L^2(Q)}$
    in the case of the smooth target $\overline{u}$ given
    by \eqref{Example 3D 1} (Example~4.2.1).}
  \label{tab:ex2smoothtar4D}
\end{table}

\begin{figure}
  \centering
  \includegraphics[width=0.8\textwidth]{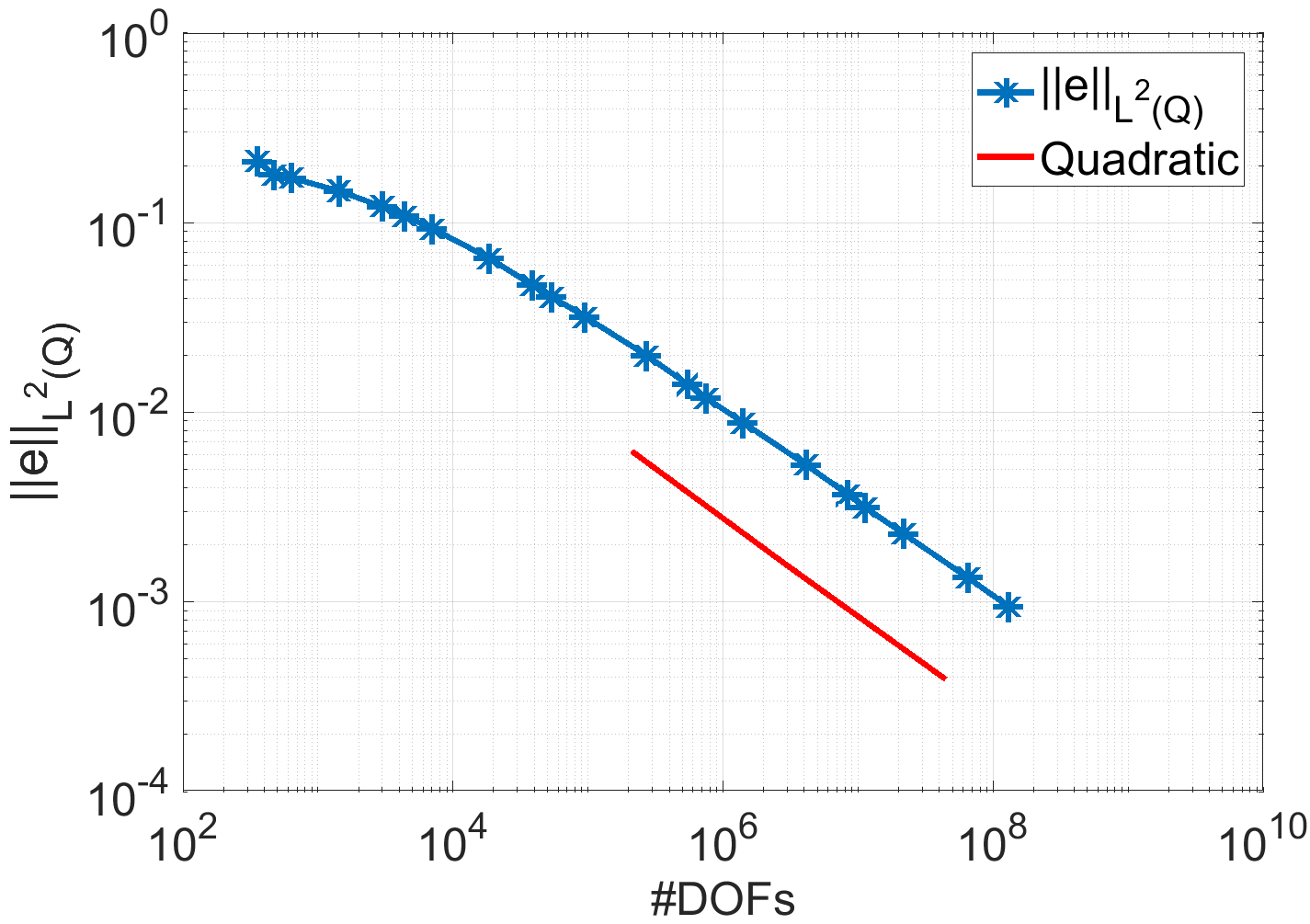}
  \caption{Error
    $\|e\|_{L^2(Q)}=\|\widetilde{u}_{\varrho h}-\overline{u}\|_{L^2(Q)}$ in
    the case of a smooth desired state
    $\overline{u}\in H_0^1(Q)\cap H^2(Q)$ in three
    space dimensions (Example~4.2.1).}   
  \label{fig:ex2smoothtar4D}
\end{figure} 

In the second example (Example~4.2.2), we take
a piecewise linear
continuous target function $\overline{u}$ being one in the mid point
$(1/2, 1/2, 1/2, 1/2)$ and zero in all corner points of $Q=(0,1)^4$.
In this case we have $\overline{u}\in X\cap H^{3/2-\varepsilon}$,
$\varepsilon > 0$, and we
observe $1.5$ as order of convergence which corresponds to the
error estimate \eqref{final estimate 1 2}, see
Table \ref{tab:ex2contar4D} and Figure~\ref{fig:ex2contar4D}.

\begin{table}
  \centering
  \begin{tabular}{|l|l|l|l|}
    \hline
    N(\#DOFs)&$h=(N/2)^{-1/4}$ & $\varrho \,(=h^2)$& $\|\widetilde{u}_{\varrho h}-\overline{u}\|_{L^2(Q)}$  \\ \hline
    $356$&$2.7378$e$-1$& $7.4953$e$-2$  & $2.1510$e$-1$ \\
    $630$&$2.3737$e$-1$& $5.6344$e$-2$  & $1.7972$e$-1$ \\    
    $2,986$&$1.6087$e$-1$& $2.5880$e$-2$ & $1.3082$e$-1$\\
    $6,930$&$1.3034$e$-1$& $1.6988$e$-2$ & $1.0510$e$-1$\\    
    $38,114$&$8.5111$e$-2$& $7.2439$e$-3$ & $6.1638$e$-2$\\    
    $94,146$&$6.7890$e$-2$& $4.6091$e$-3$ & $4.5864$e$-2$\\    
    $546,562$&$4.3737$e$-2$& $1.9129$e$-3$ & $2.4759$e$-2$\\
    $1,400,322$&$3.4570$e$-2$& $1.1951$e$-3$ & $1.7766$e$-2$\\
    $8,289,026$&$2.2163$e$-2$& $4.9121$e$-4$ & $9.2755$e$-3$\\
    $21,657,090$&$1.7432$e$-2$& $3.0389$e$-4$ & $6.5280$e$-3$\\
    $129,165,826$&$1.1155$e$-2$& $1.2444$e$-4$ & $3.3636$e$-3$\\
    \hline
  \end{tabular}
  \caption{Error $\|\widetilde{u}_{\varrho h}-\overline{u}\|_{L^2(Q)}$ in the
    case of a
    piecewise linear continuous target $\overline{u}\in
    X \cap H^{3/2-\varepsilon}(Q)$, $\varepsilon > 0$ (Example~4.2.2).}  
  \label{tab:ex2contar4D}
\end{table}

\begin{figure}
  \centering
  \includegraphics[width=0.8\textwidth]{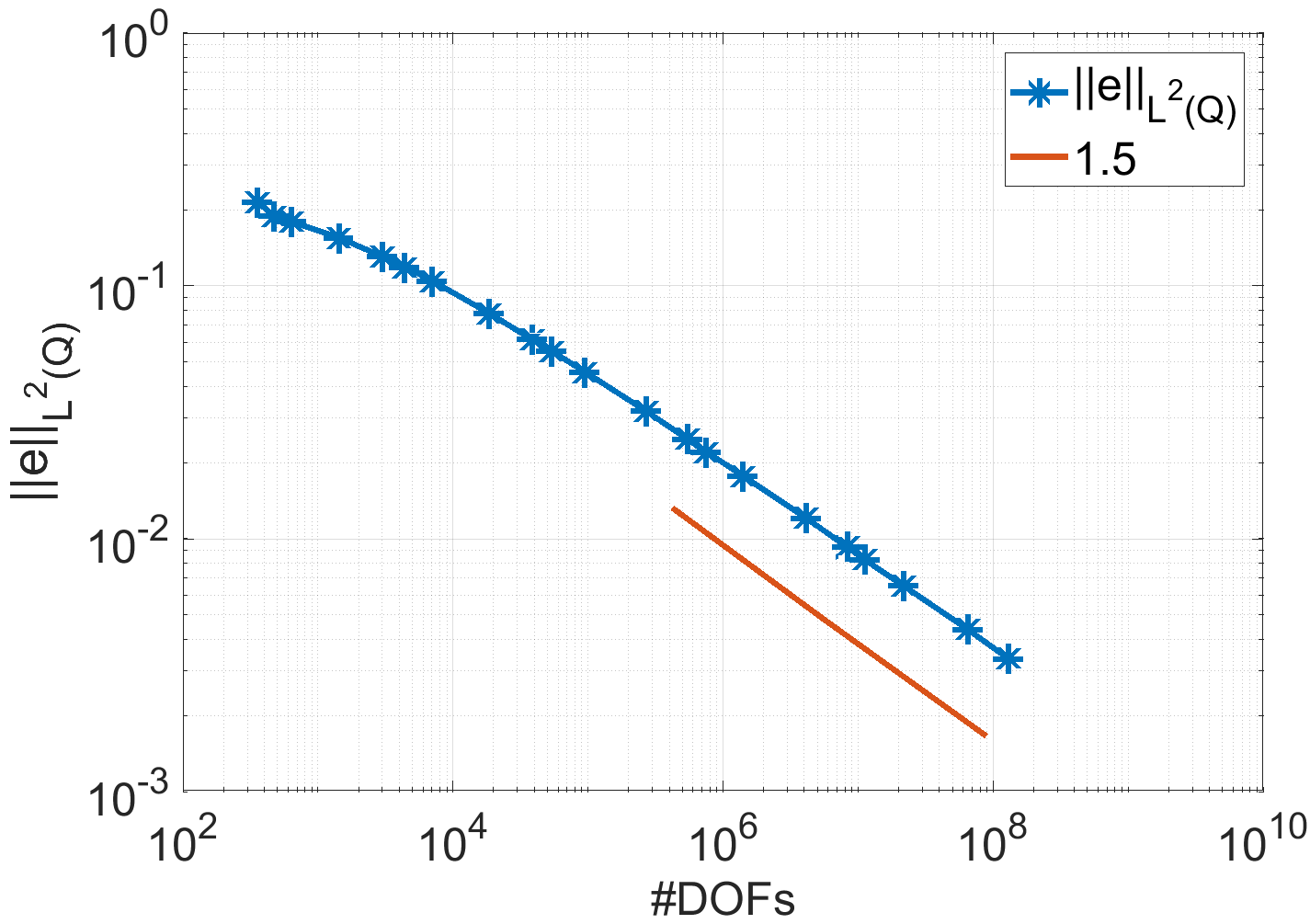}
  \caption{Error
    $\|e\|_{L^2(Q)}=\|\widetilde{u}_{\varrho h} - \overline{u}\|_{L^2(Q)}$ in
    the case of a piecewise linear continuous target $\overline{u}\in
    X \cap H^{3/2-\varepsilon}(Q)$, $\varepsilon > 0$ (Example~4.2.2).}    
  \label{fig:ex2contar4D}
\end{figure} 

In the third example (Example~4.2.3), we consider a piecewise constant discontinuous
target $\overline{u}$ which is one in the inscribed cube
$(\frac{1}{4},\frac{3}{4})^4$, and zero else. In this case we have
$\overline{u} \in H^{1/2-\varepsilon}(Q)$, $\varepsilon >0$.
From the numerical results as given in Table \ref{tab:ex2tardiscon4D}
we observe $0.5$ for the order of convergence, as expected from the
error estimate \eqref{final estimate 0 1}, see also
Figure~\ref{fig:ex2tardiscon4D}.

\begin{table}
  \centering
  \begin{tabular}{|l|l|l|l|}
    \hline
    N(\#DOFs)&$h=(N/2)^{-1/4}$ & $\varrho \,(=h^2)$& $\|\widetilde{u}_{\varrho h}-\overline{u}\|_{L^2(Q)}$  \\ \hline
    $356$&$2.7378$e$-1$& $7.4953$e$-2$  & $2.5099$e$-1$ \\
    $630$&$2.3737$e$-1$& $5.6344$e$-2$  & $1.9143$e$-1$ \\
    $2,986$&$1.6087$e$-1$& $2.5880$e$-2$ & $1.8823$e$-1$\\
    $6,930$&$1.3034$e$-1$& $1.6988$e$-2$ & $1.7500$e$-1$\\
    $38,114$&$8.5111$e$-2$& $7.2439$e$-3$ & $1.4710$e$-1$\\
    $94,146$&$6.7890$e$-2$& $4.6091$e$-3$ & $1.3313$e$-1$\\
    $546,562$&$4.3737$e$-2$& $1.9129$e$-3$ & $1.0558$e$-1$\\
    $1,400,322$& $3.4570$e$-2$ &$1.1951$e$-3$ &$9.6592$e$-2$\\
    $8,289,026$&$2.2163$e$-2$& $4.9121$e$-4$ & $7.7744$e$-2$\\
    $21,657,090$&$1.7432$e$-2$& $3.0389$e$-4$ & $6.8891$e$-2$\\
    $129,165,826$&$1.1155$e$-2$& $1.2444$e$-4$ & $5.5284$e$-2$\\
    \hline
  \end{tabular}
  \caption{Error $\|\widetilde{u}_{\varrho h}-\overline{u}\|_{L^2(Q)}$
    in the case of a
    piecewise constant and discontinuous target $\overline{u}\in
    H^{1/2-\varepsilon}(Q)$, $\varepsilon>0$ (Example~4.2.3).}  
  \label{tab:ex2tardiscon4D}
\end{table}

\begin{figure}
  \centering
  \includegraphics[width=0.8\textwidth]{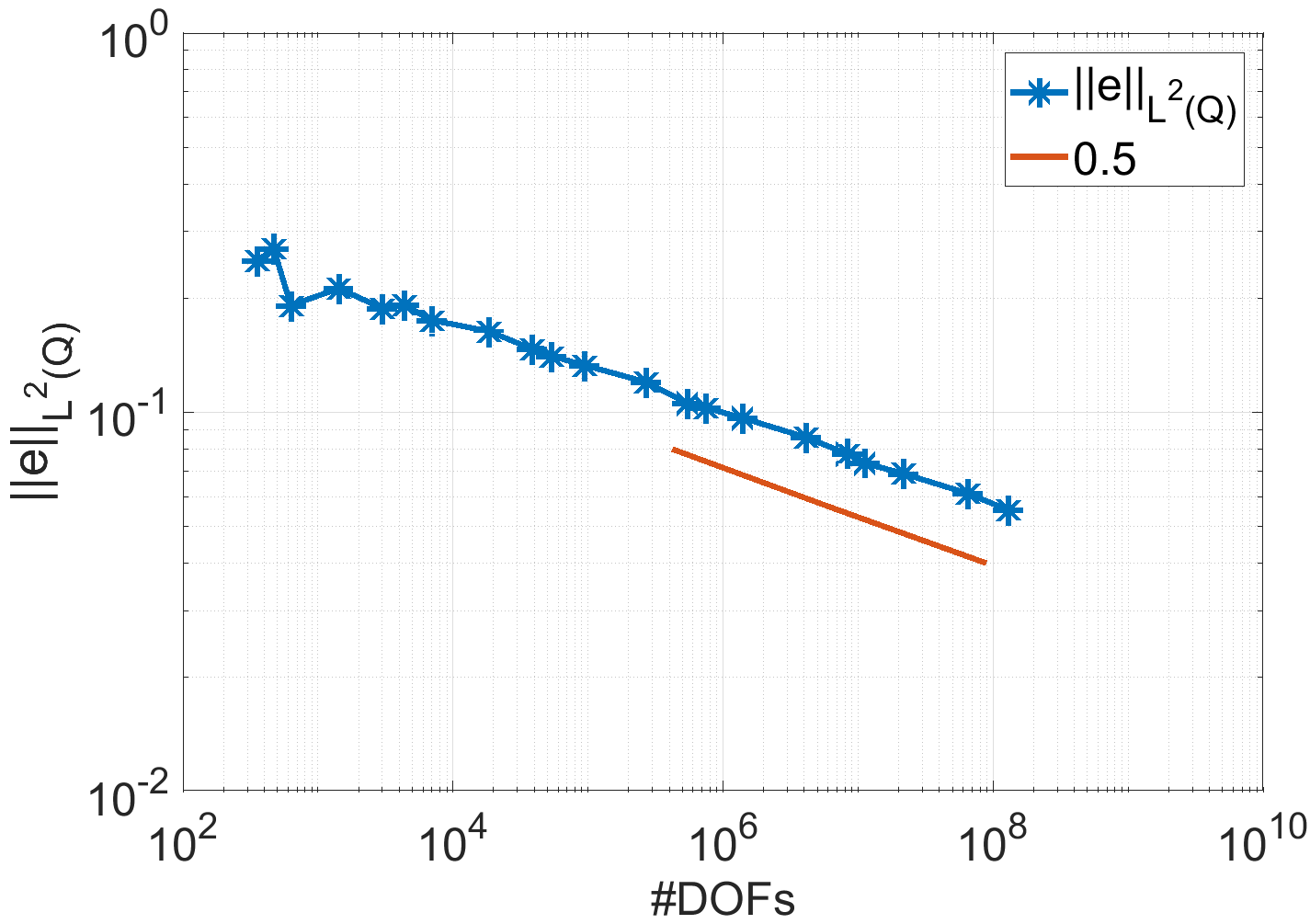}
  \caption{Error
    $\|e\|_{L^2(Q)}=\|\widetilde{u}_{\varrho h}-\overline{u}\|_{L^2(Q)}$ in
    the case of a piecewise constant and discontinuous target $\overline{u}\in
    H^{1/2-\varepsilon}(Q)$, $\varepsilon > 0$ (Example~4.2.3).}    
  \label{fig:ex2tardiscon4D}
\end{figure}


\section{Conclusions and outlook}
\label{sec:ConclusionsOutlook}
We have derived robust space-time finite element error estimates for
distributed parabolic optimal control problems with energy regularization. 
More precisely, we have estimated the $L^2(Q)$ norm of the 
error between the desired state $\overline{u}$ and the computed state
$\widetilde{u}_{\varrho h}$ depending on the regularity of the desired
state $\overline{u}$. It has been shown that the optimal convergence rate
is achieved by the proper scaling $\varrho=h^2$ between the regularization
parameter $\varrho$ and the mesh size $h$. The theoretical findings are
confirmed by several numerical examples in both two and three space
dimensions.

The theoretical results are valid for uniform mesh refinement. However, for
discontinuous targets $\overline{u}$ and targets that don't fulfil the
boundary or initial conditions, we can expect layers with steep gradients 
in the solutions as in Example~4.2.3. In this example, we have observed 
that, for a fixed $\varrho=h^2$, the adaptive version needs considerably
less unknowns to achieve the same accuracy as the corresponding uniformly
refined grid with the finest mesh-size $h$. 
Since, for adaptively refined grids, the local mesh-sizes are very different,
one can also think about a localization of the regularization parameter
$\varrho$. Another future research topic is the construction of fast and
$\varrho$ robust solvers 
for the symmetric and indefinite system \eqref{saddle-point system} 
that is equivalent to \eqref{VF parabolic S FEM pert};
see, e.g., \cite{LSY:BenziGolubLiesen:2005a,LSY:SchoeberlZulehner:2007a,LSY:SchulzWittum:2008a,LSY:Zulehner:2011a}.
Finally, the consideration of constraints imposed on the control $z_\varrho$ 
is of practical interest; see, e.g., \cite{LSY:Troeltzsch:2010a}. 


\section*{Acknowledgments}
The authors would like to acknowledge the computing support of the
supercomputer MACH--2\footnote{https://www3.risc.jku.at/projects/mach2/}
from Johannes Kepler Universit\"{a}t Linz and of the high performance
computing cluster Radon1\footnote{https://www.oeaw.ac.at/ricam/hpc}
from Johann Radon Institute for Computational and Applied Mathematics
(RICAM) on which the numerical examples are performed. 
The first and the third author were partially supported by RICAM.

\bibliography{LSY2022}
\bibliographystyle{abbrv}

\end{document}